\documentclass{article}
\usepackage{amsmath}
\usepackage{empheq}
\usepackage{mathtools}
\usepackage{amssymb}
\usepackage{mathrsfs}
\usepackage{stmaryrd}
\usepackage{bm}
\usepackage{amsfonts}
\usepackage{physics}
\usepackage{cite}
\usepackage{url}
\usepackage{hyperref}
\usepackage{enumitem}
\usepackage{amsthm}
\usepackage{thmtools}
\usepackage{nicematrix}
\usepackage{geometry}
\usepackage{spalign}
\usepackage{booktabs}
\usepackage{subcaption}
\usepackage{threeparttable}
\usepackage{multirow}
\usepackage{makecell}

\theoremstyle{plain}
\newtheorem{theorem}{Theorem}[section]
\newtheorem{lemma}[theorem]{Lemma}
\newtheorem{corollary}[theorem]{Corollary}

\theoremstyle{definition}

\newcommand\set[2]{{\left\{ {#1} \; \colon \; {#2} \right\}}}
\newcommand\setnd[1]{{\left\{ {#1} \right\}}}
\newcommand\floor[1]{{\left\lfloor {#1} \right\rfloor}}
\newcommand\ceil[1]{{\left\lceil {#1} \right\rceil}}
\newcommand\card[1]{{\left|{#1}\right|}}

\newcommand{\ZZ}{\mathbb{Z}}
\newcommand\CC{\mathbb{C}} 
 
\newcommand{\F}{\mathbb{F}}
\newcommand{\defeq}{\triangleq}
\newcommand{\dual}[1]{#1^\bot}
\newcommand{\factorial}[1]{#1!}

\DeclareMathOperator{\wt}{\mathrm{wt}}
\DeclareMathOperator{\cay}{\mathrm{Cay}}

\makeatletter
\let\oldtheequation\theequation
\renewcommand\tagform@[1]{\maketag@@@{\ignorespaces#1\unskip\@@italiccorr}}
\renewcommand\theequation{(\oldtheequation)} \makeatother

\title{On the Smallest Eigenvalues and Quantum Chromatic Numbers of
Hamming Graphs and Generalizations}
\author{
Yu Ning
\thanks{Yu Ning is with Hefei National Laboratory, Hefei 230088, China
(e-mail:sirning@mail.ustc.edu.cn)}
\and
Jack H. Koolen \thanks{Jack H. Koolen is with the School of Mathematical 
Sciences, University of Science and Technology of China, Hefei 230026, Anhui, 
China (e-mail:koolen@ustc.edu.cn)}
\and Xiande Zhang \thanks{Xiande Zhang is with the School of Mathematical
Sciences, University of Science and Technology of China, Hefei 230026, Anhui,
China, and also with Hefei National Laboratory, University of Science and
Technology of China, Hefei 230088, China (drzhangx@ustc.edu.cn)
\emph{(Corresponding author: Xiande Zhang)}}%
}

\date{}
\begin{document}
\maketitle
\begin{abstract} The smallest 
eigenvalues of (distance‑\(j\)) Hamming graphs with distance parameter \(j\) at 
least half the length were completely determined by Brouwer \emph{et al.}\ (2018). In 
the present work, we address the complementary regime, namely distances \(j\) 
strictly less than half the length, and derive asymptotic lower bounds on the 
smallest eigenvalue of 
binary Hamming graphs. For certain natural generalizations, specifically Cayley 
graphs defined over quaternary vector spaces, we asymptotically determine the 
smallest eigenvalue as well. As an application, we obtain lower bounds on the 
quantum chromatic number of these graphs. In particular, for the aforementioned 
Cayley graphs over quaternary vectors, our lower bounds for the quantum 
chromatic number coincide with known upper bounds.

\vspace{1em}
\end{abstract}

\section{Introduction}
\label{sec:introduction}
In this paper, we study the smallest eigenvalues of the adjacency matrices  of
two important  families  of  graphs, all belonging to Cayley graphs over vector
spaces.  The smallest eigenvalues of Cayley graphs over vector spaces
\cite{Feng,SDU}, especially orthogonality graphs (also known as Hadamard graphs)
\cite{SDU,ExactHadamardGraphs}, have been widely studied recently and used to
determine the chromatic number and quantum chromatic number of the corresponding
graphs,  illustrating a quantum advantage in graph coloring games
\cite{OnTheQuantumChromaticNumber,QuantumProtocol}.

\subsection{Binary Hamming Graphs}
Binary Hamming graphs are an important class of Cayley graphs on $ \ZZ_2^n$,
which are generated by a set of vectors with a fixed weight. They constitute the
most well-known  examples  of   P-polynomial  association  schemes \cite[Section
2.7]{AC}, and are important objects in graph theory \cite{DRG,DRGBook} 
and coding theory \cite{Color}. Let $H(n,j)$ denote the binary 
\emph{distance-$j$ Hamming graph},
whose vertex set is $\ZZ_2^n$, and two vectors in $ \ZZ_2^n$ are adjacent if and
only if their Hamming distance is $j$. The graph $H(n,j)$ is just a Cayley graph
on $ \ZZ_2^n$ generated by vectors of weight $j$. When $n$ is doubly even and
$j=n/2$, $H(n,j)$ is isomorphic to an orthogonality graph
\cite{ColoringOrthogonalityGraph}. The smallest eigenvalues of $H(n,j)$ are used
to study the max-cut \cite{Alon}, max-$k$-cut and chromatic number \cite{vanDam} of
certain graphs in the Hamming scheme. When $j$ is odd, $H(n,j)$ is bipartite 
and the smallest eigenvalue is $K_j(n) = -\binom{n}{j}$, where $K_j(x)$ is the 
binary Krawtchouk polynomial. When $j$ is even, it was 
conjectured \cite{vanDam} that $K_j(1) = \binom{n-1}{j} -
\binom{n-1}{j-1}$ is the smallest eigenvalue of $H(n,j)$ for even $j \ge
\frac{n+1}{2}$. This conjecture has been studied by Alon and Sudakov
\cite{Alon}, Dumer and Kapralova \cite{Dumer}, and was finally proved
 by Brouwer \emph{et al.\ }\cite{SmallestEigenvalue}
in a more general setting. However, to our knowledge, little is known about the
smallest eigenvalue of $H(n,j)$ for even $j < n/2$.

Denote $\lambda_{\min}(G)$  the smallest eigenvalue of a graph $G$. By easy 
computation, we know that $\lambda_{\min}(H(n,2))=-\lfloor \frac{n}{2}\rfloor$ 
and  $\lambda_{\min}(H(n,4))=-\frac{n^2}{4}+\Theta(n)$.\footnote{$f = O(g)$ if 
$f(n) \le c_1g(n)+c_2$ for some fixed constant $c_1,c_2 > 0$ and all possible 
$n$; $f = \Omega(g)$ if $g = O(f)$; $f = \Theta(g)$ if $f = O(g)$ and $f = 
\Omega(g)$; $f = o(g)$ if $f(n)/g(n)$ tends to $0$ as $n$ tends to $\infty$}
In this paper, we  provide asymptotic lower bounds on the smallest eigenvalue of
$H(n,j)$ for even $j < n/2$. In particular, we  show that (see
\autoref{corollary:lb_Hamming}) 
\begin{equation}\label{eq:min}
|\lambda_{\min}(H(n,j))| = 
\begin{cases}
O(n^{j/2}), & j = O(1), \\
O\qty(\left(6.563\sqrt{(1-\alpha)/\alpha}\right)^{\alpha n}), & j = \alpha n, ~0 < \alpha \le 0.17,
\end{cases}
\end{equation}
where $n$ is large enough.

\subsection{Cayley Graphs over  $\ZZ_4^n$}
The smallest eigenvalues of several classes of Cayley graphs over $\ZZ_3^n$,
which can be embedded into some orthogonality graphs, have been studied in
\cite{OUR}. Motivated by these results, the second family of graphs we consider
is the Cayley graph over $\ZZ_4^n$, denoted by $\cay(\ZZ_4^n,(r,s,r,s))$ with $n
= 2(r+s)$, which is generated by vectors of type $(r,s,r,s)$, that is, vectors
with the occurrences of symbols $0,2$ being both $r$ and the occurrences of
$1,3$ being both $s$. The edge-union of $\cay(\ZZ_4^n,(r,s,r,s))$ for all $r \in
[0,n/2]$ constitutes an orthogonality graph (see
\autoref{subsec:chiq_related_works}). There is also a graph homomorphism from
$\cay(\ZZ_4^n,(r,s,r,s))$ to $H(n,2s)$ through the quotient map $\ZZ_4^n \to
\ZZ_2^n$, and the eigenvalues of $H(n,2s)$ are related to the eigenvalues of
$\cay(\ZZ_4^n,(r,s,r,s))$. 

With an analysis of the spectrum of
$\cay(\ZZ_4^n,(r,s,r,s))$ and by the lower bounds of $\lambda_{\min}(H(n,2s))$, we
determine $\lambda_{\min}(\cay(\ZZ_4^n,(r,s,r,s))) $ for certain parameters $r$ and $s$. In fact, we show that
\begin{equation}\label{eq:rs}
\lambda_{\min}(\cay(\ZZ_4^n,(r,s,r,s))) = -\frac{\binom{n}{r,s,r,s}}{n-1}    
\end{equation}
for fixed $s\geq 2$ or $s=\frac{1}{2}\alpha n$ with $0 < \alpha \le 0.17$, where $r \ge s$ is even and large enough (see
\autoref{theorem:smallest_eigenvalue_4}).

\subsection{Quantum Chromatic Numbers}
A $c$-coloring of a graph $G$ with a vertex set $V$, is a map $V \to [c]
\defeq \setnd{1,2,\dots,c}$ such that adjacent vertices receive distinct colors
in $[c]$. The \emph{chromatic number} $\chi(G)$ of $G$ is the smallest $c$ such
that $G$ admits a $c$-coloring, which has been thoroughly studied \cite[Chapter
5]{GraphTheory}. The \emph{quantum chromatic number} 
\cite{OnTheQuantumChromaticNumber,QuantumProtocol,SpectralLB} of $G$, denoted 
by $\chi_q(G)$, is a
generalization of the chromatic number in quantum information theory
\cite{QuantumVSClassical,CostSimulating}, which is defined to be the  smallest 
$c$ such that $G$ admits a quantum $c$-coloring. 
A \emph{quantum $c$-coloring} of $G$  is a collection of orthogonal projections $P_{v,i} : \CC^d \to
\CC^d$, $v \in V$, and $i \in [c]$, satisfying that each vertex requires $\sum_{i
\in [c]} P_{v,i}=\mathbb{I}_d$, and adjacent vertices $v,w$ require
$P_{v,i}P_{w,i} = 0$ for all $i \in [c]$
\cite{OnTheQuantumChromaticNumber,SpectralLB}.
The notion of graph coloring can be restated in the framework of non-local 
games, and classical players may win the game with $\chi(G)$ colors, while 
quantum players may win with $\chi_q(G)$ colors
\cite{OnTheQuantumChromaticNumber,QuantumProtocol,SpectralLB}. Graphs satisfying
 $\chi_q(G) < \chi(G)$ illustrate a quantum advantage in graph coloring
games.

Many graphs have witnessed a gap between $\chi_q(G)$ and $\chi(G)$, especially
 orthogonality graphs $H(n,n/2)$ for doubly even $n$
\cite{OnTheQuantumChromaticNumber}, and general Hamming graphs \cite{SDU,OUR,Feng}.
The smallest eigenvalue $\lambda_{\min}(G)$ plays an important role in the
spectral lower bound on quantum chromatic numbers $\chi_q(G) \ge 1 -
\frac{d}{\lambda_{\min}(G)}$ \cite{SpectralLB}  for a regular graph with degree
$d$.

Recently, the quantum chromatic number of $H(n,j)$ for general $j$ was
studied in \cite{Feng,SDU}. When $j$ is odd,  $H(n,j)$ is a bipartite graph and thus of quantum chromatic number $2$
\cite{OnTheQuantumChromaticNumber}. For even $j$,
only the following upper bounds on $\chi_q(H(n,j))$ are known \cite{Feng,SDU,OUR}:
\begin{equation}
	\label{eqn:upper_bound_chi_q_Hamming}
\chi_q(H(n,j)) \le
\begin{cases}
2j, & j \ge n/2, \\
2\binom{n}{2}, & n/2-\sqrt{n}/2 < j < n/2, \\	
2^{h\qty(\frac{1}{2}-\sqrt{\alpha(1-\alpha)})n + o(n)},
& j = \alpha n, 0 < \alpha < 1/2,\\
n+1, &j=2; n=2^{t}, \text{or } n \equiv 3 \pmod{4} \text{ is a prime power}, \\
\end{cases}\end{equation}
where $h(x) = -x\log_2(x) - (1-x)\log_2(1-x)$ is the binary entropy function.
For lower bounds, $\chi_q(H(n,j)) \ge \frac{2j}{2j-n}$, when $j> n/2$,
\cite{SDU}. 

Using the lower bound in \autoref{eq:min}, we show that
for fixed even $j$ and $n$ large enough (see \autoref{corollary:lb_chiq_Hamming})
$$
\chi_q(H(n,j)) = \Omega(n^{j/2}).
$$
When $j < n/2$ and $j = \Theta(n)$, an
exponential lower bound on $\chi_q(H(n,j))$ is also provided (see
\autoref{corollary:lb_chiq_Hamming}).

Using the value in \autoref{eq:rs}, the lower bounds on
$\chi_q(\cay(\ZZ_4^n,(r,s,r,s)))$ are also derived, which match the known upper
bounds. In particular, we have $\chi_q(\cay(\ZZ_4^n,(r,s,r,s))) = n$ for fixed
$s\geq 2$ or $s=\frac{1}{2}\alpha n$ with $0 < \alpha \le 0.17$, where $r \ge s$
is even and large enough (see \autoref{corollary:chi4}).
 
In \autoref{tab:results} we list known results on the quantum chromatic number
of graphs in the literature.

\subsection{Organizations}
This paper is organized as follows. \autoref{sec:preliminary} introduces basic
notions, necessary results, and approaches. In \autoref{sec:lb_Hamming}, we show
an asymptotic lower bound on $\lambda_{\min}(H(n,j))$ for even $j = O(1)$ or $j
= \Theta(n)$ and $n$ large enough. \autoref{sec:Z4} is devoted to the smallest
eigenvalue of $\cay(\ZZ_4^n,(r,s,r,s))$. In \autoref{sec:qchi}, we apply the
spectral lower bound on the quantum chromatic number with the results on the smallest eigenvalues obtained in \autoref{sec:lb_Hamming} and \autoref{sec:Z4} to derive lower bounds on
$\chi_q(H(n,j))$ and $\chi_q(\cay(\ZZ_4^n,(r,s,r,s)))$, respectively. Finally, in
\autoref{sec:conclusion}, we conclude this paper and discuss some open problems.

\begin{table}[t]
\centering
\begin{threeparttable}
\caption{Known quantum chromatic numbers of graphs}
\label{tab:results}
\begin{tabular}{l|c|c}
\toprule
Graphs $G$ & $\chi_q(G)$ & References \\
\midrule
$K_n$ (the complete graph on $n$ vertices)  & $n$ & \cite{OnTheQuantumChromaticNumber} \\
graphs with $\chi(G) = k$, $k=2,3$ & $k$ & \cite{OnTheQuantumChromaticNumber} \\
$H(4t,2t) \cong \cay(\ZZ_2^{4t},(2t,2t))$ & $4t$ & \cite{ExactHadamardGraphs,Feng} \\
$\cay(\ZZ_2^{4t-1},(2t-1,2t))$ & $4t$ & \cite{Feng} \\
$\cay(\ZZ_p^{lp},(l,\dots,l))$, $p$ is an integer, $l$ large enough, $l(p-1)$
even & $lp$ & \cite{SDU} \\
$\cay(\F_q^{q^l},(q^{l-1},q^{l-1},\dots,q^{l-1}))$, $q$ is a prime power& $q^l$ & \cite{SDU} \\
\midrule
$\cay(\ZZ_3^{3l},(l,l,l))$, or $\cay(\ZZ_3^{3l-1},(l-1,l,l))$ & $3l$ & \cite{OUR} \\
\midrule
$H(n,2)$, $n \equiv 3 \pmod{4}$ is a prime power, or $n = 2^{t}$ with $t\geq 3$ & $n+1$ & \cite{OUR}
\\ 
\midrule
$H(n,j) \cong \cay(\ZZ_2^n,(n-j,j))$, $j$ even fixed, $n$ large enough &
$\Omega(n^{j/2})$ & \autoref{corollary:lb_chiq_Hamming} \\
$H(n,j) \cong \cay(\ZZ_2^n,(n-j,j))$, $j$ even, $0 < \alpha = \frac{j}{n} \le
0.17$ & $\Omega(\exp(n))$ & \autoref{corollary:lb_chiq_Hamming} \\
\midrule
$\cay(\ZZ_4^{2(r+s)},(r,s,r,s))$, $ s/n \leq 0.085$, $s\geq 2$, $r$ is even and large 
& $2(r+s)$ & \autoref{corollary:chi4} \\
\bottomrule
\end{tabular}
\begin{tablenotes}
\item[1] In the table, $\cay(\ZZ_2^{4t},(2t,2t))$ is the Cayley graph over
$\ZZ_2^{4t}$ with a generating set consisting of vectors of type $(2t,2t)$ (see
\autoref{subsec:notations}). Similar for other Cayley graphs.
\end{tablenotes}
\end{threeparttable}
\end{table}

\section{Preliminary}
\label{sec:preliminary}
For integers $m,n \in \ZZ$
with $m \le n$, let $[m,n] \defeq \setnd{m,m+1,\dots,n}$, and for $n \ge 1$, let
$[n] \defeq [1,n]$. For an integer $p$, let $\ZZ_p$ denote the
ring of integers modulo $p$.  For any real $0\leq x\leq 1$, let $h(x) \defeq -x\log_2(x) - (1-x)\log_2(1-x)$ be the binary entropy function.
\subsection{Cayley Graphs and Eigenvalues}
\label{subsec:notations}

For an abelian group $A$  and an inversion-closed subset $S \subset A$
that does not contain $0$, the Cayley graph over $A$ generated by $S$ is denoted as
$\cay(A,S)$, whose vertex set is $A$ and two vertices $v,w \in A$ are adjacent
if and only if $v-w \in S$.  For $a \in A$, we denote by $\chi_a : A \to \CC^*$ the character of
$A$ corresponding to $a$ \cite[Chapter 5]{FiniteFields}. Let
\begin{equation}
	\label{eqn:eigenvalue_Cayley_graph}
E_a  \defeq
(\chi_a(b))_{b \in A} \in \CC^A \text{ and }\lambda(a) \defeq \sum_{s \in S}
\chi_a(s).
\end{equation}
Then $\lambda(a), a \in A$ are all eigenvalues of $\cay(A,S)$ with the
corresponding  eigenvector $E_a$  \cite[Chapter 1]{SpectraCayleyGraphs}.

Now consider the Cayley graphs over $\ZZ_p^n$. For a vector $\mathbf{v} =
(v_1,\dots,v_n)\in \ZZ_p^n$  and $k \in \ZZ_p$, the \emph{$k$-th Hamming weight}
$\wt_k(\mathbf{v})$ of $\mathbf{v}$ is the occurrence of $k$ in $\mathbf{v}$,
i.e., $\wt_k(\mathbf{v}) \defeq \card{\set{i \in [n]}{v_i = k}}$. The
\emph{type} of $\mathbf{v}$ is denoted as $\mathbf{t}(\mathbf{v}) \defeq
(\wt_0(\mathbf{v}),\wt_1(\mathbf{v}),\dots,\wt_{p-1}(\mathbf{v}))$, which is a
non-negative ordered partition of $n$.  For a type $\mathbf{t} =
(t_0,\dots,t_{p-1})$, let $\ZZ_p(\mathbf{t})$ denote the set of vectors of type
$\mathbf{t}$ in $\ZZ_p^n$. Then $\card{\ZZ_p(\mathbf{t})} =
\binom{n}{\mathbf{t}} \triangleq \binom{n}{t_0,\dots,t_{p-1}} =
\frac{\factorial{n}}{\factorial{(t_0)}\factorial{(t_1)}\dots
\factorial{(t_{p-1})}}$. If $\ZZ_p(\mathbf{t})$ is inversion-closed and does not
contain the all-zero vector $\mathbf{0}$, we write $\cay(\ZZ_p^n,\mathbf{t}) \defeq
\cay(\ZZ_p^n, \ZZ_p(\mathbf{t}))$ for the Cayley graph over $\ZZ_p^n$ generated
by the set of vectors of type $\mathbf{t}$.

To compute the eigenvalues of the Cayley graph $\cay(\ZZ_p^n,\mathbf{t})$, by \autoref{eqn:eigenvalue_Cayley_graph}, we need all characters of
$\ZZ_p^n$. It is known \cite[Chapter 5]{FiniteFields} that all the characters of $\ZZ_p^n$ are

\begin{equation}
\label{eqn:character_p}
\chi_{\mathbf{v}} : \ZZ_p^n \to \CC^*, \quad \mathbf{x} \mapsto
\zeta_p^{\mathbf{v} \cdot \mathbf{x}},
\end{equation}
where $\mathbf{v} \in \ZZ_p^n$ and $\zeta_p \defeq e^{2 \pi \sqrt{-1} / p}$ is the
$p$-th root of unity. Here $\mathbf{v} \cdot \mathbf{x}$ is the inner product
modulo $p$. So,
\begin{equation}
\label{eqn:lambda}\lambda(\mathbf{v})=\sum_{\mathbf{x}\in \ZZ_p(\mathbf{t})} \zeta_p^{\mathbf{v} \cdot \mathbf{x}}.
\end{equation}
\subsection{Weight Enumerators of Codes}

A $\ZZ_p$-code $C$ is a $\ZZ_p$-submodule of $\ZZ_p^n$.
The \emph{(complete) weight enumerator} of $C$ is a polynomial  defined as
$$
A_C(x_0,x_1,\dots,x_{p-1}) =
\sum_{\mathbf{c} \in C}\prod_{i \in \ZZ_p}x_i^{\wt_i(\mathbf{c})}.
$$
For simplicity, we write $f[\mathbf{t}]$ for the coefficient of
$x_0^{t_0}x_1^{t_1}\dots x_{p-1}^{t_{p-1}}$ in a polynomial
$f(x_0,\dots,x_{p-1})$. Then
$A_C[\mathbf{t}]$ is exactly the number of codewords
of type $\mathbf{t}=(t_0,\ldots,t_{p-1})$ in $C$.
The dual code $\dual{C}$ of $C$ is defined as
$$
\dual{C} \defeq \set{\mathbf{w} \in \ZZ_p^n}{\forall ~\mathbf{v} \in C,
\mathbf{w} \cdot \mathbf{v} = 0},
$$
which is again a $\ZZ_p$-code.

Now, we consider the weight enumerator of codes generated by a single vector.
That is, for $\mathbf{v} \in \ZZ_p^n$, let
$$
\langle\mathbf{v}\rangle_p \defeq \set{a\mathbf{v}}{a \in \ZZ_p}
$$
be the code generated by $\mathbf{v}$. It is easy to see that both  enumerators
$A_{\langle\mathbf{v}\rangle_p}$ and $A_{\dual{\langle\mathbf{v}\rangle_p}}$
depend only on the type $\mathbf{t} = \mathbf{t}(\mathbf{v})$ of $\mathbf{v}$,
so we write $A_{\mathbf{t}} \defeq A_{\langle\mathbf{v}\rangle_p}$ and
$A_{\dual{\mathbf{t}}} \defeq A_{\dual{\langle\mathbf{v}\rangle_p}}$. The
following is a simple generalization of \cite[Lemma 1]{OUR} from $p=3$ to every
$p$.

\begin{lemma}[\cite{OUR}]
	\label{lemma:duality}
Let $\mathbf{s}$ and $\mathbf{t}$ be two types of vectors in $\ZZ_p^n$. Then
\begin{equation}
	\label{eqn:duality}
\binom{n}{\mathbf{s}} \cdot A_{\dual{\mathbf{s}}}[\mathbf{t}] =
\binom{n}{\mathbf{t}} \cdot A_{\dual{\mathbf{t}}}[\mathbf{s}].
\end{equation}
\end{lemma}
\begin{proof}
Consider the bipartite graph on vertices in $\ZZ_p(\mathbf{s}) \cup
\ZZ_p(\mathbf{t})$, where $\mathbf{v} \in \ZZ_p(\mathbf{s})$ and $\mathbf{w} \in
\ZZ_p(\mathbf{t})$ are adjacent if and only if $\mathbf{v} \cdot \mathbf{w} =
0$. The degree of $\mathbf{v}$ in this bipartite graph is exactly
$A_{\dual{\mathbf{s}}}[\mathbf{t}]$, and the degree of $\mathbf{w}$ is exactly
$A_{\dual{\mathbf{t}}}[\mathbf{s}]$. Therefore, by counting the number of edges, we have 
$$
\card{\ZZ_p(\mathbf{s})} \cdot A_{\dual{\mathbf{s}}}[\mathbf{t}] = \card{\ZZ_p(\mathbf{t})} \cdot A_{\dual{\mathbf{t}}}[\mathbf{s}].
$$
The proof is completed by noting that $\card{\ZZ_p(\mathbf{s})} = \binom{n}{\mathbf{s}}$ and
$\card{\ZZ_p(\mathbf{t})} = \binom{n}{\mathbf{t}}$.
\end{proof}
When $p=2$, let $\mathbf{s}=(n-i,i)$ and $\mathbf{t}=(n-j,j)$. Then applying \autoref{lemma:duality} with simple computations gives the well-known identity
\begin{equation}
	\label{eqn:duality_K}
\binom{n}{i} K_j(i) = \binom{n}{j}K_i(j),
\end{equation}
where $K_j(x)$ is the Krawtchouk polynomial \cite{KrawtchoukPolynomialsLevenshtein} defined by
\begin{equation}
	\label{eqn:generating_function_K}
K_j(x) \defeq \sum_{i=0}^j (-1)^i\binom{x}{i}\binom{n-x}{j-i}= \qty((y+z)^{n-x}(y-z)^x)[y^{n-j}z^j].
\end{equation}

The weight enumerators of a code $C$ and its dual $\dual{C}$ are related by the
MacWilliams identity \cite[Section 2.3.5]{SelfDual}. We state them for $p=2,4$ for
later use. When $p = 2$, we have \cite{macwilliams1977theory}
\cite{Wan2000}, 
\begin{equation}
	\label{eqn:MacWilliams_2}
A_{\dual{C}}(x_0,x_1) = \frac{1}{\card{C}}A_C(x_0+x_1,x_0-x_1),
\end{equation}
and for $p = 4$, we have \cite{Wan2000},\cite[Section 2.3.5]{SelfDual} 
\begin{multline}
\label{eqn:MacWilliams_4}
A_{\dual{C}}(x_0,x_1,x_2,x_3) =
\frac{1}{\card{C}}A_C(x_0+x_1+x_2+x_3,
x_0+\zeta_4 x_1 - x_2 - \zeta_4 x_3, \\
x_0-x_1+x_2-x_3,
x_0-\zeta_4 x_1 - x_2 + \zeta_4 x_3).
\end{multline}

\section{Bounds on the Smallest Eigenvalue of Hamming Graphs}
\label{sec:lb_Hamming}

This section is devoted to the proof of \autoref{eq:min}, the lower bound on $\lambda_{\min}(H(n,j))$ for even
$j < n/2$. In fact, we will show that for large $n$ (see 
\autoref{corollary:lb_Hamming}),
$$
|\lambda_{\min}(H(n,j)) |= 
\begin{cases}
O\left(n^{j/2}\right), & j = O(1), \\
O\qty(\left(e(1+\sqrt{2})\sqrt{(1-\alpha)/\alpha}\right)^{\alpha n}), & j = \alpha n, ~0 < \alpha \le 0.17.
\end{cases}
$$

Note that $H(n,j) = \cay(\ZZ_2^n,(n-j,j))$.
 For $\mathbf{w} \in \ZZ_2^n$ with Hamming
weight $w = \wt(\mathbf{w})$,  by \autoref{eqn:eigenvalue_Cayley_graph},
\begin{equation}
	\label{eqn:common_ev}
E_\mathbf{w} = (\chi_{\mathbf{w}}(\mathbf{v}))_{\mathbf{v} \in \ZZ_2^n} = ((-1)^{\mathbf{w}
\cdot \mathbf{v}})_{\mathbf{v} \in \ZZ_2^n} \in \CC^{\ZZ_2^n}
\end{equation}
 is an eigenvector of
$H(n,j)$ with the corresponding eigenvalue \cite{SmallestEigenvalue},
\begin{equation}
	\label{eqn:Krawtchouk}
\lambda(\mathbf{w}) = \sum_{\mathbf{v} \in \ZZ_2(n-j,j)}(-1)^{\mathbf{w}\cdot \mathbf{v}}
= \sum_{i=0}^j(-1)^i\binom{w}{i}\binom{n-w}{j-i} = K_j(w).
\end{equation}
So $\lambda_{\min}(H(n,j)) =K_j(w)$ for a specific $w \in [0,n]$. If $j$ is odd,
$H(n,j)$ is bipartite and $\lambda_{\min}(H(n,j))$ $ = -\binom{n}{j}=K_j(n)$. For
even $j \ge n/2$, it was shown in \cite{SmallestEigenvalue} that
\begin{equation}
	\label{eqn:Brouwer}
\lambda_{\min}(H(n,j)) = \begin{cases}
	K_j(1), & j > n/2, \\
	K_j(2), & j = n/2.
\end{cases}
\end{equation}
For the case where $j < n/2$ is even, it will be difficult to determine for
which $w \in [0,n]$, $K_j(w)$ is the smallest for general $j$. But we know more
when $j$ is small.

When $j=2$, it is known \cite{OUR} that
$$
\lambda_{\min}(H(n,2)) = \begin{cases}
K_2(n/2) = -n/2, & n \text{ even}, \\
K_2((n+1)/2) = K_2((n-1)/2) = -(n-1)/2, & n \text{ odd}.
\end{cases}
$$
When $j = 4$, $K_j(x)$ is a polynomial of degree $4$, and 
reaches the minimum at $x_0 = \left(n-\sqrt{3n-4}\right)/2$ and 
$x_1 = \left(n+\sqrt{3n-4}\right)/2$, by analyzing the derivative of $K_j(x)$. 
Also note that $K_j(x)$ is symmetric relative to $x = n/2$ for even $j$. 
Therefore, 
$$\lambda_{\min}(H(n,4)) =
\min \qty(K_4(\floor{x_0}),K_4(\ceil{x_0})) =
\min \qty(K_4(\floor{x_1}),K_4(\ceil{x_1})).
$$ 
In particular,
$\lambda_{\min}(H(n,4))\ge K_4(x_0) = K_4(x_1) = -n^2/4 + 3n/4 - 2/3$, and
$\lambda_{\min}(H(n,4))$  is close to $-n^2/4 + 3n/4 - 2/3$.

Our goal is to give lower bounds on $K_j(w)$ for $w \in [0,n]$ and even $j< n/2$.

\subsection{Roadmap}
In this subsection, we explain the approach to the lower bounds on $K_j(w)$ for
$w \in [0,n]$ and even $j < n/2$. We will use some notions from
the theory of distance-regular graphs \cite{DRG,DRGBook}.

A connected graph $G$ with diameter $D$ is called \emph{distance-regular} if
there are constants $c_i,a_i,b_i$, the so-called \emph{intersection numbers},
such that for all $i\in [0,D]$ and all vertices $u$ and $v$ at distance $i$,
among the neighbors of $v$, there are $c_i$ at distance $i-1$ from $u$, $a_i$ at
distance $i$, and $b_i$ at distance $i + 1$. Then $H(n,1)$ is a distance-regular
graph of diameter $n$ and of intersection numbers \[\text{$c_i = i$, $a_i = 0$ and $b_i = n-i$ with $i\in [0,n]$}.\]In the following, we always refer to $a_i,b_i$ and $c_i$ as the above values.

Let $A_j, j \in [1,n]$ be the adjacency matrix of
$H(n,j)$, and $A_0$ be the identity matrix of size $2^n$. Then
$A_j$ satisfies the following recursion \cite[Page 13]{DRG}
\cite[Section 4.1]{DRGBook},
\begin{equation}
\label{eqn:recurrence_A}
A_1A_j = c_{j+1}A_{j+1} + b_{j-1}A_{j-1}.
\end{equation}
In accordance to \autoref{eqn:recurrence_A}, for $j \in [0,n]$, we consider 
polynomials $q_j(x)$ defined by the recursion,
\begin{equation}
\label{eqn:recurrence_q}
\begin{cases*}
xq_j(x) = c_{j+1}q_{j+1}(x) + b_{j-1}q_{j-1}(x), \\
q_0(x) = 1, q_1(x) = x.
\end{cases*}
\end{equation}
So $q_j(x)$ is a polynomial of degree $j$ satisfying $q_j(A_1) = A_j$, for all
$j \in [0,n]$ \cite[Page 13]{DRG}\cite[Section 4.1]{DRGBook}.
We note that $q_j(x)$ is related to $K_j(x)$ as in the
following lemma.
\begin{lemma}
	\label{lemma:qt_and_K}For all $j \in [0,n]$,
$q_j(n-2x) = K_j(x)$.
\end{lemma}
\begin{proof}
For $\mathbf{w} \in \ZZ_2^n$, by \autoref{eqn:common_ev} and
\autoref{eqn:Krawtchouk}, $A_jE_{\mathbf{w}} = K_j(w)E_{\mathbf{w}}$ where $w =
\wt(\mathbf{w})$. On the other hand,
$$
A_jE_{\mathbf{w}} = q_j(A_1)E_\mathbf{w} = q_j(K_1(w))E_\mathbf{w}.
$$
Note that $K_1(w) = n-2w$, so we have $q_j(n-2w) = K_j(w)$ for all integers $w
\in [0,n]$. As $q_j(n-2x)$ and $K_j(x)$ are both polynomials of degree $j \le n$
and  share the same value at $n+1$ points, we conclude that $q_j(n-2x) =
K_j(x)$ for all $x$.
\end{proof}
By \autoref{lemma:qt_and_K}, finding the smallest $K_j(w)$ for $w \in [0,n]$ is
equivalent to finding the smallest $q_j(n-2w)$. Thus,
\begin{equation}
	\label{eqn:lb_Hamming_overview_0}
	\lambda_{\min}(H(n,j)) \ge \min_{-n \le x \le n} q_{j}(x).
\end{equation}
Next, we compute  all coefficients of $q_j(x)$. Before that, we need  the following notation.
A set $I = \setnd{u_1,\dots,u_k}$ of integers
with $u_1 < u_2 < \dots < u_k$ is \emph{$2$-separated} if $u_{i+1} -
u_i \ge 2$ for all $i \in [k-1]$. 
Let
$\mathcal{I}(k,j)$ be the set of all $k$-subsets of $[0,j-2]$ that are
$2$-separated.
For $I \in \mathcal{I}(k,j)$, let $$R(I) \defeq \prod_{i \in I}b_ic_{i+1}=\prod_{i \in I}(n-i)(i+1)$$ and
$R(\emptyset) \defeq 1$. 
With these notations, we determine  the coefficients of $q_j(x)$ below.
\begin{theorem}
	\label{theorem:coefficients} 
For any $j\in [0,n]$, let $L_i(j)$ be the coefficient of the term $x^{j-i}$ in $q_j(x)$, $i\in [0,j]$.
Then 
\begin{equation}
L_i(j) = \begin{cases}
0, & i \equiv 1 \pmod{2}, \\
\frac{(-1)^{i/2}}{j!} \sum_{I \in \mathcal{I}(i/2,j)}R(I), & i \equiv 0 \pmod{2}.
\end{cases}
\end{equation}
\end{theorem}
The proof of \autoref{theorem:coefficients} will be provided in
\autoref{subsection:proof_coefficients}. Although we have a formula for all
coefficients of $q_j(x)$ in \autoref{theorem:coefficients}, they are too complicated to
determine the minimum value of $q_j(x)$.

From now on, we assume that $j<n/2$ is even. 
Let $r_0,r_1$ be the smallest and largest root of $q_j(x)$, respectively. As
$q_j(x)$ is of degree $j$ and positive leading coefficient, we have
$q_j(x) \ge 0$ for $x \le r_0$ or $x \ge r_1$.
Then \autoref{eqn:lb_Hamming_overview_0} implies that
\begin{equation}
	\label{eqn:lb_Hamming_overview}
	\lambda_{\min}(H(n,j)) \ge \min_{r_0 < x < r_1} q_{j}(x).
\end{equation}
The bound on $r_0,r_1$ can be obtained from the  roots of
$K_j(x)$. It is known \cite[Eq. (126)]{KrawtchoukPolynomialsLevenshtein} that $K_j(x)$ has $j$ simple roots $0 < x_{1} < x_{2} < \dots
<x_{j} < n$, and for $j < n/2$, the smallest
root $x_{1}$ is bounded by
\begin{equation}
	\label{eqn:smallest_zero_K}
x_{1} \ge \frac{n}{2} - \frac{1}{2}\qty(\sqrt{(j-1)(n-j+2)} + \sqrt{(j-2)(n-j+3)}).
\end{equation}
As $K_j(x) = (-1)^jK_j(n-x)$, the roots of $K_j(x)$ are symmetric with respect
to $n/2$. By \autoref{lemma:qt_and_K}, $r_0 = n-2x_{j}$ and $r_1 = n-2x_{1}$. In
particular,
\begin{equation}
	\label{eqn:r0r1}
-r_0 = r_1 \le \sqrt{(j-1)(n-j+2)} + \sqrt{(j-2)(n-j+3)}.
\end{equation}

Next, we estimate $q_j(x)$ and provide the desired lower bounds on the smallest
eigenvalue of $H(n,j)$ by combining \autoref{theorem:coefficients},
\autoref{eqn:lb_Hamming_overview} and \autoref{eqn:r0r1}.

\subsection{Lower Bounds on the Smallest Eigenvalue}
The lower bounds are derived by considering the magnitude of $r_0,r_1$ and
the coefficients of $q_j(x)$. First, we consider the regime when $j$ is fixed and
$n$ is large enough, namely, $j = O(1)$.
\begin{theorem}
	\label{theorem:lb_q} For fixed even $j$ and large enough $n$, we have $\abs{q_{j}(x)} = O(n^{j/2})$ when $r_0 \leq x \leq r_1$.
\end{theorem}
\begin{proof}  
Since each element in $I \in \mathcal{I}(k,j)$ is less than $j$, when $j$ is fixed
and $n$ is large, $R(I)=\prod_{i \in I}(n-i)(i+1) = O(n^{k})$. By
\autoref{theorem:coefficients}, $\abs{L_{2k}(j)} = O(n^k)$. By
\autoref{eqn:r0r1}, $\abs{x} = O(\sqrt{n})$ for $r_0 \leq x \leq r_1$ when $j$
is fixed. Therefore, 
$$
\abs{q_{j}(x)} \le \sum_{k=0}^{j/2} \abs{L_{2k}(j)} \cdot \abs{x}^{j-2k}
= O(n^{j/2}).
$$
\end{proof}

The second regime we consider is $j = \Theta(n)$.
\begin{theorem}
	\label{theorem:lb_q_theta} 
Let $j=\alpha n$ be even with $0 < \alpha < 1/3$. Then for $r_0 \leq x \leq
r_1$,
$$\abs{q_{j}(x)} = \frac{(C_\alpha n)^{j}}{\factorial{j}} \cdot
\qty(\frac{(1+\sqrt{2})^{j+1}}{2\sqrt{2}}+o(1)),
$$
where $C_\alpha=\sqrt{\alpha(1-\alpha)}$.
\end{theorem}
\begin{proof} For each $i\leq j-2$, we have  $(n-i)(i+1)\leq (n-(j-2))((j-2)+1) = (n-\alpha n+2)(\alpha n-1)\leq (1-\alpha)\alpha n^2$ since $0 < \alpha < 1/3$. So, first we have  $\abs{x} \leq 2C_\alpha n$ for $r_0 \leq x \leq r_1$ by \autoref{eqn:r0r1}. Second, 
for $I \in \mathcal{I}(k,j)$, we have $R(I) =
\prod_{i \in I}(n-i)(i+1) \leq (C_\alpha n)^{2k}$.
Since $\card{\mathcal{I}(k,j)} = \binom{j-k}{k}$, 
we have $\abs{L_{2k}(j)} \leq \binom{j-k}{k} \qty(C_\alpha n)^{2k}/j!$.

We conclude that
\begin{align}
\abs{q_{j}(x)} &\le \sum_{k=0}^{j/2}{\abs{L_{2k}(j)}} \cdot \abs{x}^{j-2k}
\leq \frac{(2C_\alpha n)^j}{\factorial{j}} 
\sum_{k=0}^{j/2}\binom{j-k}{k}\qty(\frac{1}{4})^k
\\
&=\frac{(2C_\alpha n)^{j}}{\factorial{j}} \cdot
\frac{(1+\sqrt{2})^{j+1} -
(1-\sqrt{2})^{j+1}}{2^{j+1}\sqrt{2}} \label{eqn:by_Fibonacci}  \\ 
&=\frac{(C_\alpha n)^{j}}{\factorial{j}} \cdot \qty(\frac{(1+\sqrt{2})^{j+1}}{2\sqrt{2}}+o(1)),
\label{eqn:by_Fibonacci1} 
\end{align}
where \autoref{eqn:by_Fibonacci} follows from \autoref{eqn:f2} in
\autoref{appsec:Fibonacci}.
\end{proof}

\begin{corollary}
	\label{corollary:lb_Hamming}
Let $j < n/2$ be even and let $n$ be large enough. Then
$$
|\lambda_{\min}(H(n,j)) |= 
\begin{cases}
O\left(n^{j/2}\right), & j = O(1), \\
O\qty(\left(e(1+\sqrt{2})\sqrt{(1-\alpha)/\alpha}\right)^{\alpha n}), & j = \alpha n,~ 0<\alpha< 1/3.
\end{cases}
$$
\end{corollary}
\begin{proof}
By \autoref{eqn:lb_Hamming_overview}, we have $|\lambda_{\min}(H(n,j)) |\le
\max_{r_0 < x < r_1} \abs{q_j(x)}$. The result follows immediately from
\autoref{theorem:lb_q}, \autoref{theorem:lb_q_theta}, and 
Stirling's approximation \cite{Stirling}.
\end{proof}

There is a trivial bound $|\lambda_{\min}(H(n,j))| \le \binom{n}{j}$. 
\autoref{corollary:lb_Hamming} provides a great improvement on this 
trivial bound for fixed $j$. We compare them for $j = \alpha n$, where $\binom{n}{j} =
\Theta\qty(\frac{2^{nh(\alpha)}}{\sqrt{n}})$. 
For \autoref{corollary:lb_Hamming} to provide a better
 bound on $|\lambda_{\min}(H(n,j))|$ than the trivial one, it is necessary that 
$$
\qty(e(1+\sqrt{2})\sqrt{(1-\alpha)/\alpha})^j < 2^{nh(\alpha)}
$$ for large enough $n$,
which is equivalent to
\begin{equation}
	\label{eqn:region_alpha}
	\alpha \log_2\qty(e(1+\sqrt{2})\sqrt{(1-\alpha)/\alpha}) < h(\alpha).
\end{equation}
By the numerical method, \autoref{eqn:region_alpha} holds for 
$\alpha \le 0.17$ and does not hold for $\alpha > 0.18$. In other words, 
\autoref{corollary:lb_Hamming} is not trivial for $\alpha \le 0.17$.

\subsection{Proof of \autoref{theorem:coefficients}}
\label{subsection:proof_coefficients}
We finish \autoref{sec:lb_Hamming} with a proof of
\autoref{theorem:coefficients}. Instead of determining $L_i(j)$ for $i\in
[0,j]$, we consider $M_i(j) \defeq c_jc_{j-1}\dots c_1
L_i(j)=\factorial{j}L_i(j)$, where $c_k=k$ is the intersection number of
$H(n,1)$. Then it is enough to show that  for any  given $j \in [1,n]$ and   $i\in
[0,j]$,
\begin{equation}
	\label{eqn:coefficients}
M_i(j) = \begin{cases}
0, & i \equiv 1 \pmod{2}, \\
(-1)^{i/2} \sum_{I \in \mathcal{I}(i/2,j)}R(I), & i \equiv 0 \pmod{2}.
\end{cases}
\end{equation}
The proof is by an induction on $i$ and $j$. We
check the base cases where $i=0,1$ or $i=j$ first.
\begin{lemma}
	\label{lemma:s01} Given $j \in [1,n]$, 
$M_0(j) = 1$, and  $M_1(j) = 0$. 
\end{lemma}
\begin{proof}
By \autoref{eqn:recurrence_q}, we have $c_{j+1}q_{j+1}(x) =
xq_j(x) - b_{j-1}q_{j-1}(x)$. Comparing the coefficients of $x^{j+1}$ on both sides, we have
$c_{j+1}L_0(j+1) = L_0(j)$. Multiplying both sides by $c_jc_{j-1}\dots c_1$
gives $M_0(j+1) = M_0(j)$. As $q_1(x) = x$ and $c_1 = 1$, we have $M_0(1) = 1$.
Thus, $M_0(j) = 1$ for all $j \ge 1$.

Similarly, comparing the coefficients of $x^{j}$, we have $c_{j+1}L_1(j+1) =
L_1(j)$, and thus $M_{1}(j+1) = M_1(j)$. As $M_1(1) = 0$, we have $M_1(j) = 0$
for all $j \ge 1$.
\end{proof}

Next, we consider the case $i=j$.

\begin{lemma}
	\label{lemma:Mtt}
 Given $j \in [1,n]$, 
\begin{equation}
	\label{eqn:Mtt}
M_j(j) = \begin{cases}
	0, & j \equiv 1 \pmod{2}, \\
	(-1)^{j/2}\prod_{k=0}^{j/2-1}b_{2k}c_{2k+1}, & j \equiv 0 \pmod{2}.
\end{cases}
\end{equation}
\end{lemma}
\begin{proof}
Recall that $c_1 = 1$, $q_0(x) = 1$, and $c_1q_1(x) = q_1(x) = x$. By
\autoref{eqn:recurrence_q}, $c_2c_1q_2(x) = c_2q_2(x) = x^2-b_0$. We conclude
that $M_1(1) = 0$, $M_2(2) = -b_0 = -b_0c_1$, and \autoref{eqn:Mtt} holds for $j
= 1,2$. In the following, we show \autoref{eqn:Mtt} for $j \ge 3$ by 
induction on $j$. By \autoref{eqn:recurrence_q}, we have
$$
M_{j}(j) = - b_{j-2}c_{j-1}M_{j-2}(j-2).
$$
As $M_1(1) = 0$, we conclude that $M_j(j) = 0$ for odd $j$ .
For even $j$, by the induction hypothesis,
$$
M_j(j) = -b_{j-2}c_{j-1}(-1)^{j/2-1}\prod_{k=0}^{j/2-2}b_{2k}c_{2k+1} =
(-1)^{j/2} \prod_{k=0}^{j/2-1}b_{2k}c_{2k+1}.
$$
We conclude that \autoref{eqn:Mtt} holds for all $j\ge 1$.
\end{proof}
It can be verified that \autoref{lemma:s01} and \autoref{lemma:Mtt} coincide
with \autoref{eqn:coefficients} for the case $i = 0,1$ or $i=j$.
\begin{proof}[Proof of \autoref{theorem:coefficients}]
The proof is by inductions on $i$ and $j$. First, we carry out an induction on
$i$ for any $j$. The base cases when $i=0,1$ for any $j\in [1,n]$ are assured
by \autoref{lemma:s01}. Suppose $i \ge 2$ and the induction hypothesis on $i$
states that \autoref{eqn:coefficients} holds for any $i' < i$ and all $ j \in
[i',n]$. Now we show that \autoref{eqn:coefficients} holds for  $i$ and all $j
\in [i,n]$. Next we take an induction on $j$ where the base case $j=i$ is
assured by \autoref{lemma:Mtt}. Assume that given $i$,
\autoref{eqn:coefficients}  holds for $i\le j' \le j$, and we prove it for
$j+1$.

 Considering coefficients of $x^{j+1-i}$ on both sides of \autoref{eqn:recurrence_q}, we have
$$
c_{j+1}L_i(j+1) = L_i(j) - b_{j-1}L_{i-2}(j-1).
$$
Multiplying both sides by $c_jc_{j-1}\dots c_1$ gives
\begin{equation}\label{eq:mij1}
    M_i(j+1) = M_i(j) - b_{j-1}c_jM_{i-2}(j-1).
\end{equation}
By the induction hypothesis on $j$ for given $i$, we have 
\begin{equation}
	\label{eqn:mij}
M_i(j) = \begin{cases}
0, & i \equiv 1 \pmod{2}, \\
(-1)^{i/2} \sum_{I \in \mathcal{I}(i/2,j)}R(I), & i \equiv 0 \pmod{2}.
\end{cases}
\end{equation}
By the induction hypothesis on $i-2$ for $j-1\in [i-2,n]$, we have 
\begin{equation}
	\label{eqn:mi2j}
M_{i-2}(j-1) = \begin{cases}
0, & i \equiv 1 \pmod{2}, \\
(-1)^{i/2-1} \sum_{I \in \mathcal{I}(i/2-1,j-1)}R(I), & i \equiv 0 \pmod{2}.
\end{cases}
\end{equation}

Now, we compute $M_i(j+1)$ by \autoref{eq:mij1}, \autoref{eqn:mij} and
\autoref{eqn:mi2j}. When $i$ is odd, it is easy to get $M_i(j+1)=0.$ Assume that
$i$ is even. By the definition of  $\mathcal{I}(k,j)$, $I\in
\mathcal{I}(i/2-1,j-1)$ if and only if $I\cup \{j-1\} \in \mathcal{I}(i/2,j+1)$.
That is, there is a one-to-one correspondence between members in $
\mathcal{I}(i/2-1,j-1)$ and members in $\mathcal{I}(i/2,j+1)$ containing the
element $j-1$. Further by the definition of $R(I)$,   we have
\begin{equation}\label{eq:rij}
    - b_{j-1}c_jM_{i-2}(j-1)=(-1)^{i/2}  b_{j-1}c_j \sum_{I \in \mathcal{I}(i/2-1,j-1)}R(I)=(-1)^{i/2} \sum_{\stackrel{J \in \mathcal{I}(i/2,j+1)}{j-1 \in J}}R(J).
\end{equation}
Note that $\mathcal{I}(i/2,j+1)\setminus \mathcal{I}(i/2,j)$  consists of all members in $\mathcal{I}(i/2,j+1)$ containing the element $j-1$. So by \autoref{eq:mij1} and \autoref{eq:rij}, we have 

$$M_i(j+1)=(-1)^{i/2} \sum_{I \in \mathcal{I}(i/2,j)}R(I)+ (-1)^{i/2} \sum_{\stackrel{J \in \mathcal{I}(i/2,j+1)}{j-1 \in J}}R(J)=(-1)^{i/2} \sum_{J \in \mathcal{I}(i/2,j+1)}R(J).$$
Thus, \autoref{eqn:coefficients}  holds for  $j+1$ and given $i$. The induction 
proof is completed.
\end{proof}

\section{The Smallest Eigenvalue of Cayley Graphs over $\ZZ_4^n$}
\label{sec:Z4}
In this section, we asymptotically determine the smallest eigenvalue of
$\cay(\ZZ_4^n,(r,s,r,s))$, the Cayley graph over $\ZZ_4^n$ generated by vectors
in $\ZZ_4^n$ of type $(r,s,r,s)$ with $n = 2(r+s)$ (see
\autoref{subsec:notations}). For simplicity of notations, we write $G(r,s)
\defeq \cay(\ZZ_4^n,(r,s,r,s))$ and $S \defeq \ZZ_4(r,s,r,s)$.

\subsection{The Expression of Eigenvalues of $G(r,s)$}
\label{subsec:ev4}
In this subsection, we express all the eigenvalues of $G(r,s)$ with weight
enumerators.
 
For any  $\mathbf{v} \in \ZZ_4^n$, the corresponding eigenvalue of $G(r,s)$ is
$\lambda(\mathbf{v}) = \sum_{\mathbf{b} \in S}\zeta_4^{\mathbf{b} \cdot
\mathbf{v}} $ by \autoref{eqn:lambda}. We partition $S$ into four disjoint
subsets,
$$
S(\mathbf{v},a) \defeq \set{\mathbf{b} \in S}{\mathbf{b} \cdot \mathbf{v} = a},~ a\in \ZZ_4.
$$
Note that there is a bijection $S(\mathbf{v},1) \to S(\mathbf{v},3), \;
\mathbf{x} \mapsto -\mathbf{x}.$ So, 
\begin{align}
\lambda(\mathbf{v}) &= \sum_{\mathbf{b} \in S}\zeta_4^{\mathbf{b} \cdot \mathbf{v}} =
\card{S({\mathbf{v}},0)} + \card{S({\mathbf{v}},1)}\zeta_4
+ \card{S({\mathbf{v}},2)}\zeta_4^2 + \card{S({\mathbf{v}},3)}\zeta_4^3 \\
&= \card{S(\mathbf{v},0)} - \card{S(\mathbf{v},2)}=2\card{S(\mathbf{v},0)} - 
\beta({\mathbf{v}}), \label{eqn:ev4}
\end{align}
where $\beta(\mathbf{v}) \defeq \card{S(\mathbf{v},0)} + \card{S(\mathbf{v},2)}$.
We will determine $\beta({\mathbf{v}})$ and $\card{S(\mathbf{v},0)}$ separately.

For $\mathbf{v} = (v_1,\dots,v_n) \in \ZZ_4^n$, we write $\overline{\mathbf{v}}
\defeq (\overline{v_1},\dots,\overline{v_n}) \in \ZZ_2^n$ as the image of
$\mathbf{v}$ through the quotient map $\ZZ_4^n \to \ZZ_2^n$, i.e.,
$\overline{v_i} \defeq v_i \pmod{2} \in \ZZ_2$ for $v_i \in \ZZ_4$.
\begin{lemma}
	\label{lemma:beta_v}
For $\mathbf{v} \in \ZZ_4^n$ with $\mathbf{t}(\mathbf{v}) = (t_0,t_1,t_2,t_3)$, we have
\begin{equation}
\label{eqn:beta_2}
\beta(\mathbf{v}) = \frac{\binom{n}{r,s,r,s}}{\binom{n}{t_0+t_2}}
\qty(\frac{1}{2}(x+y)^n + \frac{1}{2}(x+y)^{2r}(x-y)^{2s})[t_0+t_2,t_1+t_3]. 
\end{equation}
\end{lemma}
\begin{proof}
First, we show that 
\begin{equation}
	\label{eqn:beta_1}
\beta(\mathbf{v}) = A_{\dual{\langle\overline{\mathbf{v}}\rangle_2}}(x,y)[2r,2s]
\binom{2r}{r}\binom{2s}{s}.
\end{equation}
For $\mathbf{b} \in S({\mathbf{v}},0) \cup S({\mathbf{v}},2)$, it is clear that
$\overline{\mathbf{b}} \cdot \overline{\mathbf{v}} = 0 \in \ZZ_2$. That is, for
each $\mathbf{b}  \in S({\mathbf{v}},0) \cup S({\mathbf{v}},2)$ with
$\mathbf{t}(\mathbf{b}) = (r,s,r,s)$, we have $\mathbf{t}(\overline{\mathbf{b}})
= (2r,2s)$ and $\overline{\mathbf{b}}$ is a codeword of type $(2r,2s)$ in
$\dual{\langle\overline{\mathbf{v}}\rangle_2}$. Conversely, for  each
$\overline{\mathbf{b}}  \in \dual{\langle\overline{\mathbf{v}}\rangle_2}$  of
type $(2r,2s)$, there are $\binom{2r}{r}\binom{2s}{s}$ of preimages
$\mathbf{b} \in S(\mathbf{v},0) \cup S(\mathbf{v},2)$ of type $(r,s,r,s)$. This
number is obtained by replacing in $\overline{\mathbf{b}}$ half of the $1$ with
$3$, and half of the $0$ with $2$. So \autoref{eqn:beta_1} is established.

As $\mathbf{t}(\overline{\mathbf{v}}) = (t_0+t_2,t_1+t_3)$.
By \autoref{lemma:duality}, we have
$$
\binom{n}{t_0+t_2}A_{\dual{\langle\overline{\mathbf{v}}\rangle_2}}[2r,2s] =
\binom{n}{2r}A_{\dual{(2r,2s)}}[t_0+t_2,t_1+t_3].
$$
For different $s$, we have 
\begin{equation}
	\label{eqn:A2}
A_{{(2r,2s)}}(x,y)  = \begin{cases}
	x^n, & s=0, \\
	x^n + x^{n-2s}y^{2s}, & s\neq 0. 
\end{cases}
\end{equation}
By the MacWilliams identity (\autoref{eqn:MacWilliams_2}),
\begin{equation}
	\label{eqn:temp}
A_{\dual{(2r,2s)}}(x,y) = \frac{1}{2}(x+y)^n + \frac{1}{2}(x+y)^{2r}(x-y)^{2s}.
\end{equation}
Note that $\binom{n}{r,s,r,s} = \binom{n}{2r}\binom{2r}{r}\binom{2s}{s}$. Then
\autoref{eqn:beta_2} follows.
\end{proof}
Next, we give an expression of $\card{S(\mathbf{v},0)}$.

\begin{lemma}
	\label{lemma:Lrv0}
For $\mathbf{v} \in \ZZ_4^n$ with $\mathbf{t}(\mathbf{v}) = (t_0,t_1,t_2,t_3)$, we have
\begin{multline}
\card{S(\mathbf{v},0)} = 
\frac{\binom{n}{r,s,r,s}}{4\binom{n}{t_0,t_1,t_2,t_3}}
\big((x+y+z+w)^n
+ 2((x+z)^2-(y+w)^2)^r((x-z)^2+(y-w)^2)^s \\
+(x+y+z+w)^{2r}(x-y+z-w)^{2s}\big)[t_0,t_1,t_2,t_3].
\end{multline}
\end{lemma}
\begin{proof}
By definition and \autoref{lemma:duality}, 
$$
\card{S(\mathbf{v},0)} = A_{\dual{(t_0,t_1,t_2,t_3)}}[r,s,r,s]
= \frac{\binom{n}{r,s,r,s}}{\binom{n}{t_0,t_1,t_2,t_3}}
A_{\dual{(r,s,r,s)}}[t_0,t_1,t_2,t_3].
$$
Note that 
$$
A_{(r,s,r,s)}(x,y,z,w) = \begin{cases}
x^n + 2x^ry^sz^rw^s + x^{2r}z^{2s}, & s \neq 0, \\
x^n + x^{n/2}z^{n/2}, & s = 0.
\end{cases}
$$
By the MacWilliams identity (\autoref{eqn:MacWilliams_4}),
\begin{multline}
	\label{eqn:Lrv0_1}
A_{\dual{(r,s,r,s)}}(x,y,z,w) = \frac{1}{4}(x+y+z+w)^n
+ \frac{1}{2}((x+z)^2-(y+w)^2)^r((x-z)^2+(y-w)^2)^s \\
+\frac{1}{4}(x+y+z+w)^{2r}(x-y+z-w)^{2s}.
\end{multline}
\end{proof}

Now  we  are able to  give an expression of $\lambda(\mathbf{v})$.

\begin{theorem}
	\label{theorem:eigenvalue_4}
Let $n$ be even, $r \in [0,n/2]$ and $s = n/2-r$. For $\mathbf{v} \in \ZZ_4^{n}$
with $\mathbf{t}(\mathbf{v}) = (t_0,t_1,t_2,t_3)$, the eigenvalue
$\lambda(\mathbf{v})$ of $G(r,s) = \cay(\ZZ_4^n,(r,s,r,s))$ is given by
\begin{equation}
\label{eqn:lambda_v_43}
\lambda(\mathbf{v}) = \frac{\binom{n}{r,s,r,s}}{\binom{n}{t_0,t_1,t_2,t_3}}
\qty(((x+z)^2-(y+w)^2)^r((x-z)^2+(y-w)^2)^s)[t_0,t_1,t_2,t_3].
\end{equation}
\end{theorem}
\begin{proof}
By \autoref{lemma:beta_v}, \autoref{lemma:Lrv0} and \autoref{eqn:ev4},
\begin{align}
\lambda(\mathbf{v}) &= 2\card{S(\mathbf{v},0)} - \beta(\mathbf{v}) \notag \\
&= \frac{\binom{n}{r,s,r,s}}{\binom{n}{t_0,t_1,t_2,t_3}}\big(
	((x+z)^2-(y+w)^2)^r((x-z)^2+(y-w)^2)^s \notag \\
& \quad \quad \quad \quad \quad \quad + \frac{1}{2}(x+y+z+w)^{2r}(x-y+z-w)^{2s}
\big)[t_0,t_1,t_2,t_3]  \notag \\
& \quad- \frac{\binom{n}{r,s,r,s}}{2\binom{n}{t_0+t_2}}\qty(
	(x+y)^{2r}(x-y)^{2s})[t_0+t_2,t_1+t_3]  \label{eqn:lambda_v_42}.
\end{align}
We show that \autoref{eqn:lambda_v_42} and \autoref{eqn:lambda_v_43} are equal in the following.

Let $f(x,y,z,w) = (x+y+z+w)^{2r}(x-y+z-w)^{2s}$ and $g(x,y) = (x+y)^{2r}(x-y)^{2s}$.
Note that $\binom{n}{t_0,t_1,t_2,t_3} = \binom{n}{t_0+t_2}\binom{t_0+t_2}{t_0}
\binom{t_1+t_3}{t_1}$. Then, it suffices to show that
\begin{equation}
\label{eqn:lambda_v_4_suffice}	
f[t_0,t_1,t_2,t_3] = \binom{t_0+t_2}{t_0}\binom{t_1+t_3}{t_1}g[t_0+t_2,t_1+t_3].
\end{equation}

We first compute $f[t_0,t_1,t_2,t_3]$. Since there are $2r$ terms $(x+y+z+w)$
and $2s$ terms $(x-y+z-w)$ in $f$, we use a pair of vectors $(\mathbf{u},
\mathbf{w}) \in \ZZ_4^{2r} \times \ZZ_4^{2s}= \ZZ_4^{n}$ to indicate which of
$\{x,y,z,w\}$ is chosen from each term by a one-to-one correspondence between
$\ZZ_4$ and $\{x,y,z,w\}$. For example,  $u_i=0$ means that $x$ is chosen from
the $i$-th $(x+y+z+w)$ and $w_j=1$ means that $-y$ is chosen from the $j$-th
$(x-y+z-w)$. Note that a choice mode contributes to $f[t_0,t_1,t_2,t_3]$ if and
only if the vector $(\mathbf{u}, \mathbf{w}) \in  \ZZ_4^{n}$ is of type
$(t_0,t_1,t_2,t_3)$. So
\begin{equation}
\label{eqn:tmp_4_1}
f[t_0,t_1,t_2,t_3] = \sum_{\mathbf{u},\mathbf{w}} (-1)^{\wt_1(\mathbf{w})+\wt_3(\mathbf{w})},
\end{equation}
where the summand $(\mathbf{u},\mathbf{w})$ runs over $\ZZ_4(t_0,t_1,t_2,t_3)$.

Similarly for $g[t_0+t_2,t_1+t_3]$, we use a vector  $(\mathbf{x},\mathbf{y}) \in \ZZ_2^{2r} \times
\ZZ_2^{2s} = \ZZ_2^{n}$ to indicate a way of choices of  $x,y$ 
from each factor of $g(x,y)$. Then
\begin{equation}
\label{eqn:tmp_4_2}
g[t_0+t_2,t_1+t_3] = \sum_{\mathbf{x},\mathbf{y}} (-1)^{\wt(\mathbf{y})},
\end{equation}
where the summand $(\mathbf{x},\mathbf{y})$ runs over $\ZZ_2(t_0+t_2,t_1+t_3)$.

Finally, note that the quotient map $\ZZ_4^n \to \ZZ_2^n$ restricts to
$\ZZ_4(t_0,t_1,t_2,t_3) \to \ZZ_2(t_0+t_2,t_1+t_3)$, and each element in
$\ZZ_2(t_0+t_2,t_1+t_3)$ has $\binom{t_0+t_2}{t_0}\binom{t_1+t_3}{t_1}$
preimages. Moreover, $(-1)^{\wt_1(\mathbf{w}) + \wt_3(\mathbf{w})} =
(-1)^{\wt(\mathbf{y})}$ when $\mathbf{w}$ is a preimage of $\mathbf{y}$. So we
conclude that \autoref{eqn:lambda_v_4_suffice} is true.
\end{proof}

\subsection{The Smallest Eigenvalue of $G(r,s)$}
\label{subsec:smallest_ev4}
In this section, we asymptotically determine the smallest eigenvalue of $G(r,s)
= \cay(\ZZ_4^n,(r,s,r,s))$ with $n = 2(r+s)$ for even $r \ge s$ in suitable
regimes.

First, we list some observations and notations that are convenient for the
discussion. By \autoref{theorem:eigenvalue_4}, for $\mathbf{v} \in \ZZ_4^n$ with
$\mathbf{t}(\mathbf{v}) = (t_0,t_1,t_2,t_3)$, the eigenvalue
$\lambda(\mathbf{v})$ of $G(r,s)$ depends only on the type of $\mathbf{v}$, so
we write $\lambda(t_0,t_1,t_2,t_3) \defeq \lambda(\mathbf{v})$. As the
polynomial in \autoref{eqn:lambda_v_43} for the expression of
$\lambda(\mathbf{v})$ is symmetric with respect to the transposition of $x,z$
and the transposition of $y,w$, we know that $\lambda(t_0,t_1,t_2,t_3)$ is
independent of the order of $t_0,t_2$ and the order of $t_1,t_3$. So, we may
assume $t_0 \le t_2$ and $t_1 \le t_3$. For $\mathbf{b} \in S \defeq
\ZZ_4(r,s,r,s)$, $\mathbf{1} \cdot \mathbf{b} = 2r \in \ZZ_4$, where $\mathbf{1}
\in \ZZ_4^n$ is the all-one vector. Thus,
\begin{equation}
	\label{eqn:cyclic_invariant_4}
\lambda(\mathbf{v} + \mathbf{1}) = \sum_{\mathbf{b} \in S}
\zeta_4^{(\mathbf{v}+\mathbf{1}) \cdot \mathbf{b}} = (-1)^r\lambda(\mathbf{v}).
\end{equation}
As $\mathbf{t}(\mathbf{v} + \mathbf{1}) = (t_3,t_0,t_1,t_2)$, if $r$ is even, we have
$\lambda(t_0,t_1,t_2,t_3) = \lambda(t_3,t_0,t_1,t_2)$, and we can further assume
that $t_0 \le t_1$. Note that the largest eigenvalue of $G(r,s)$ is
$\lambda_{\max}(G(r,s)) = \lambda(0,0,0,n) = \binom{n}{r,s,r,s}$. 

In summary, we can always assume that $t_0 \le t_2$, $t_0\leq t_1 \le t_3$, and
$(t_0,t_1,t_2,t_3) \notin \{(0,0,0,n),$ $(0,0,n,0)\}$ in the following.

\begin{theorem}
	\label{theorem:smallest_eigenvalue_4}
Let $n = 2(r+s)$ with $r \ge s$ and $r$ be even.
The smallest eigenvalue of $G(r,s)$ is given by
\begin{equation}
\label{eqn:smallest_ev_expr}	
\lambda_{\min}(G(r,s)) = \lambda(0,1,n-2,1) = -\frac{\binom{n}{r,s,r,s}}{n-1},
\end{equation}
in both of the following two regimes
\begin{enumerate}
	\item[(R1).] $s \ge 2$ is fixed and $n$ is large enough, i.e. $s = O(1)$;
	\item[(R2).] $s=\alpha \cdot \frac{n}{2}$ and $n$ is large enough, where $0
	< \alpha \le 0.17$.
\end{enumerate}
\end{theorem}
The proof of Theorem~\ref{theorem:smallest_eigenvalue_4} will be given in the following steps. 
\begin{itemize}
\item The expression for $\lambda(0,1,n-2,1)$ in
Eq.~\eqref{eqn:smallest_ev_expr} is verified in \autoref{lemma:smallest_ev_expr}.
\item When $t_0+t_2 = 0$ or $t_1+t_3 = 0$, the result $\lambda(t_0,t_1,t_2,t_3)
 > \lambda(0,1,n-2,1)$ is proved in \autoref{lemma:ev_from_Z2} for the two regimes
 (R1) and (R2).
\item When $t_0+t_2  \neq 0$ and $t_1+t_3  \neq 0$, the result
$\abs{\lambda(t_0,t_1,t_2,t_3)} \le \abs{\lambda(0,1,n-2,1)}$ is proved in
\autoref{lemma:abs_ub_4}.
\end{itemize}

\begin{lemma}
	\label{lemma:smallest_ev_expr}
Let $n = 2(r+s)$ with $r \ge s$. For $G(r,s)$, $\lambda(0,1,n-2,1) =
-\frac{\binom{n}{r,s,r,s}}{n-1}.$
\end{lemma}
\begin{proof}
By \autoref{theorem:eigenvalue_4},
\begin{align*}
\lambda(0,1,n-2,1)
&= \frac{\binom{n}{r,s,r,s}}{\binom{n}{0,1,n-2,1}}
\qty(((x+z)^2 - (y+w)^2)^r((x-z)^2+(y-w)^2)^s)[0,1,n-2,1] \\
&= 	\frac{\binom{n}{r,s,r,s}}{\binom{n}{0,1,n-2,1}}
\qty((z^2-2yw)^r(z^2-2yw)^s)[0,1,n-2,1] \\
&= \frac{\binom{n}{r,s,r,s}}{\binom{n}{0,1,n-2,1}}(-2(r+s)) \\
&= -\frac{\binom{n}{r,s,r,s}}{n-1},
\end{align*}
where the second equality follows from the fact that we cannot choose any terms
containing $x$, $y^2$ or $w^2$ to get the corresponding coefficient.
\end{proof}

 In \autoref{lemma:ev_from_Z2}, we only prove the case where $t_0+t_2 = 0$. The
 case  $t_1+t_3 = 0$ is a consequence of \autoref{lemma:ev_from_Z2} by
 $\lambda(t_0,t_1,t_2,t_3) = \lambda(t_3,t_0,t_1,t_2)$.

\begin{lemma}
	\label{lemma:ev_from_Z2}
Let $n = 2(r+s) = t_1+t_3$ with $t_1\leq t_3$. Then, the eigenvalue $\lambda(0,t_1,0,t_3)$ of $G(r,s)$
is given by
$$
\lambda(0,t_1,0,t_3) = (-1)^r\frac{\binom{n}{r,s,r,s}}{\binom{n}{2s}} K_{2s}(t_3).
$$
In particular, $\lambda(t_0,t_1,t_2,t_3) > \lambda(0,1,n-2,1)$ when $r$ is even in both regimes (R1) and (R2).
\end{lemma}
\begin{proof}
By \autoref{theorem:eigenvalue_4},
\begin{align}
&\lambda(0,t_1,0,t_3) \notag\\
&= \frac{\binom{n}{r,s,r,s}}{\binom{n}{0,t_1,0,t_3}}
\qty(((x+z)^2-(y+w)^2)^r((x-z)^2+(y-w)^2)^s)[0,t_1,0,t_3] \label{eqn:temp40}\\
&= (-1)^r\frac{\binom{n}{r,s,r,s}}{\binom{n}{0,t_1,0,t_3}}
\qty((y+w)^{2r}(y-w)^{2s})[0,t_1,0,t_3]
\label{eqn:temp41}  \\
&= (-1)^r\frac{\binom{n}{r,s,r,s}}{\binom{n}{t_3}}K_{t_3}(2s) \label{eqn:temp42}\\
&= (-1)^r\frac{\binom{n}{r,s,r,s}}{\binom{n}{2s}} K_{2s}(t_3). \label{eqn:temp43}
\end{align}
We explain how these equations are derived. \autoref{eqn:temp41} follows from
the fact that we cannot choose any terms containing $x$ or $z$, to get the
corresponding term $y^{t_1}w^{t_3}$. 
\autoref{eqn:temp42} and \autoref{eqn:temp43} follow
from \autoref{eqn:generating_function_K} and \autoref{eqn:duality_K}.

In the following, assume $r$ is even. Note that $\lambda(0,t_1,0,t_3) >
\lambda(0,1,n-2,1)$ is equivalent to 
\begin{equation}
\label{eqn:ev_from_Z2_1}	
-K_{2s}(t_3) < \frac{\binom{n}{2s}}{n-1}.
\end{equation} Recall that $K_{2s}(t_3)$ is one of the
eigenvalues of $H(n,2s)$ (see \autoref{eqn:Krawtchouk}). If $K_{2s}(t_3) \ge 0$,
then \autoref{eqn:ev_from_Z2_1} holds trivially. So we assume $K_{2s}(t_3) < 0$,
and thus $-K_{2s}(t_3) \le |\lambda_{\min}(H(n,2s))|$.

Consider the regime (R1), that is, $s \ge 2$ is fixed and $n$ is large enough. By
\autoref{corollary:lb_Hamming}, $-K_{2s}(t_3) \le |\lambda_{\min}(H(n,2s))| = 
O(n^s)$ while $\frac{\binom{n}{2s}}{n-1} = \Theta(n^{2s-1})$. So
\autoref{eqn:ev_from_Z2_1} holds in (R1).

Consider (R2). We have $\frac{\binom{n}{2s}}{n-1} = \Theta\qty(
\frac{2^{nh(\alpha)}}{n\sqrt{n}})$, where $\alpha = \frac{2s}{n}$. Similarly,
by \autoref{corollary:lb_Hamming},
$$
-K_{2s}(t_3) \le |\lambda_{\min}(H(n,2s))| = 
O\qty(\left(e(1+\sqrt{2})\sqrt{(1-\alpha)/\alpha}\right)^j).
$$
So we have 
$$
-K_{2s}(t_3) \le |\lambda_{\min}(H(n,2s))| < \frac{\binom{n}{2s}}{n-1},
$$
when \autoref{eqn:region_alpha} holds (in particular, when $0 < \alpha \le
0.17$) and $n$ is large enough.
\end{proof}

Note that in the proof of \autoref{lemma:ev_from_Z2}, when $r$ is odd, we need
to change Eq. \eqref{eqn:ev_from_Z2_1} to $K_{2s}(t_3) <
\frac{\binom{n}{2s}}{n-1}.$ However, we cannot upper bound $K_{2s}(t_3)$ by
$|\lambda_{\min}(H(n,2s))|$ as $K_{2s}(t_3)$ might be positive and larger than
$|\lambda_{\min}(H(n,2s))|$. 
\begin{lemma}
	\label{lemma:abs_ub_4}
Let $n = 2(r+s)$ and $(t_0,t_1,t_2,t_3)$ be a non-negative partition of $n$.
Suppose $t_0+t_2 \neq 0$, $t_1+t_3 \neq 0$ and $n \ge 10$. Then
$$
\abs{\lambda(t_0,t_1,t_2,t_3)} \le
\frac{\binom{n}{r,s,r,s}}{n-1} = \abs{\lambda(0,1,n-2,1)}.
$$
\end{lemma}

\begin{proof}
By \autoref{theorem:eigenvalue_4},
$$
\abs{\lambda(t_0,t_1,t_2,t_3)} \le
\frac{\binom{n}{r,s,r,s}}{\binom{n}{t_0,t_1,t_2,t_3}}((x+z)^2 + (y+w)^2)^{r+s}[t_0,t_1,t_2,t_3].
$$
We expand the polynomial
\begin{align}
((x+z)^2+(y+w)^2)^{n/2} &= \sum_{k=0}^{n/2} \binom{n/2}{k} (x+z)^{2k} (y+w)^{n-2k} \\
&= \sum_{k=0}^{n/2} \binom{n/2}{k} \qty(
	\sum_{l=0}^{2k} \binom{2k}{l} x^{2k-l}{z^l}
)
\qty(
	\sum_{m=0}^{n-2k} \binom{n-2k}{m} y^{n-2k-m}w^m
) \\
&= \sum_{k=0}^{n/2}\sum_{l=0}^{2k}\sum_{m=0}^{n-2k}
\binom{n/2}{k}\binom{2k}{l}\binom{n-2k}{m} x^{2k-l}z^l y^{n-2k-m}w^m.
\label{eqn:expansion_4}
\end{align}Let
$$
f(t_0,t_1,t_2,t_3) \defeq ((x+z)^2+(y+w)^2)^{n/2}[t_0,t_1,t_2,t_3].
$$
Clearly, if $t_0+t_2$ is odd or $t_1+t_3$ is odd, we have $f(t_0,t_1,t_2,t_3) = 0$. So, we assume that
both $t_0+t_2$ and $t_1+t_3$ are even. From \autoref{eqn:expansion_4}, we have
$$
f(t_0,t_1,t_2,t_3) = \binom{n/2}{(t_0+t_2)/2}\binom{t_0+t_2}{t_0}\binom{t_1+t_3}{t_1}.
$$
Then $$\abs{\lambda(t_0,t_1,t_2,t_3)}
\le \frac{\binom{n}{r,s,r,s}}{\binom{n}{t_0,t_1,t_2,t_3}}f(t_0,t_1,t_2,t_3) = \frac{\binom{n}{r,s,r,s}\binom{n/2}{(t_0+t_2)/2} }{\binom{n}{t_0+t_2}}.$$

To show $\abs{\lambda(t_0,t_1,t_2,t_3)} \le
\frac{\binom{n}{r,s,r,s}}{n-1} = \abs{\lambda(0,1,n-2,1)}$, it suffices to show that
$$
\binom{n}{t_0+t_2}  \ge (n-1)\binom{n/2}{(t_0+t_2)/2}.
$$
As $(t_0+t_2) \notin \setnd{0,n}$ and $t_0+t_2$ is even, we have $(t_0+t_2) \in [2,n-2]$.
It suffices to show that
\begin{equation}
\label{eqn:goal_4}
\forall~ m \in [1,n/2-1], \, \binom{n}{2m} \ge (n-1) \binom{n/2}{m}.
\end{equation}
We note the following inequality
$$
\binom{n}{2m} = \sum_{j=0}^{2m} \binom{n/2}{j} \binom{n/2}{2m-j} \ge
\binom{n/2}{m}\binom{n/2}{m}.
$$
As $\binom{n/2}{m} \ge \binom{n/2}{2}$ for all $m \in [2,n/2-2]$, and
$\binom{n/2}{2} \ge n-1$ for all even $n \ge 10$, we have
$$
\binom{n}{2m} \ge \binom{n/2}{m}\binom{n/2}{m} \ge \binom{n/2}{2}\binom{n/2}{m}
\ge (n-1)\binom{n/2}{m}.
$$
for all even $n \ge 10$ and $m \in [2,n/2-2]$.
For the rest case where $m \in \setnd{1,n/2-1}$, \autoref{eqn:goal_4} holds trivially.
\end{proof}

\section{Quantum Chromatic Numbers}
\label{sec:qchi}
We first introduce some related works on quantum chromatic numbers.

\subsection{Related Works}
\label{subsec:chiq_related_works}
The precise
definition of quantum coloring and the quantum chromatic number $\chi_q(G)$ of a
graph $G$, can be found in \autoref{sec:introduction}.

For the lower bound of $\chi_q(G)$, we have the following spectral lower bound in
\cite{SpectralLB} 
\begin{equation}
	\label{eqn:spectral_lb}
\chi_q(G) \ge 1 - \frac{\lambda_{\max}(G)}{\lambda_{\min}(G)}.
\end{equation}
By this bound and \autoref{eqn:Brouwer} we have (also see
\cite[Theorem 1.3]{SDU}),

\begin{equation}
	\label{eqn:lower_bound_chi_q_Hamming}
\chi_q(H(n,j)) \ge
\begin{cases}
\frac{2j}{2j-n}, & \text{ even } j > n/2, \\
n, &  j = n/2 \text{ and $4\mid n$}.
\end{cases}
\end{equation}  
However, there is no lower bound of $\chi_q(H(n,j))$ for even $j<n/2$. When $j$
is odd, we have $\chi_q(H(n,j)) = 2$ \cite{OnTheQuantumChromaticNumber}.

The upper bounds of $\chi_q(G)$ can be derived from flat orthogonal
representations of $G$ \cite{Oddities}. Orthogonality graphs (also known as
Hadamard graphs) \cite{ExactHadamardGraphs} constitute a family of graphs
naturally equipped with a flat orthogonal representation. Let $O_{n,p}$ be the
\emph{orthogonality graph}, whose vertices are vectors in $\CC^n$ with entries
consisting of only the $p$-th roots of unity, and two vertices are adjacent in
$O_{n,p}$ if and only if they are orthogonal. This representation of vertices
yields $\chi_q(O_{n,p}) \le n$  \cite{Oddities}. There is a natural isomorphism
between $O_{n,p}$ and the Cayley graphs $\cay(\ZZ_p^n,(l,l,\dots,l))$ when
$n=lp$ and $p$ is a prime. So, this determines $\chi_q(O_{n,2}) =
\chi_q(H(n,n/2)) = n$ for $n$ divided by $4$. Recently, Cao \emph{et al.}
\cite{SDU} extended this result to $\chi_q(O_{lp,p}) = lp$ for any integer $p
\ge 2$ and $l$ large enough such that $l(p-1)$ is even.  All these exact values
are obtained by the upper bound from orthogonal representations and the spectral
lower bound.

The upper bounds of $\chi_q(G)$ can also be derived from  graph homomorphisms
\cite{OnTheQuantumChromaticNumber}. If there is a graph homomorphism $G \to H$,
then
\cite{OnTheQuantumChromaticNumber}
\begin{equation}\label{eqn:subg}
  \chi_q(G) \le \chi_q(H).
\end{equation}
For example, when $j > n/2$, there is a graph homomorphism
$H(n,j) \to H(2j,j)$ whose underlying map on the vertices is given by
$$
\ZZ_2^n \to \ZZ_2^{2j}, \quad (v_1,\dots,v_n) \mapsto (v_1,\dots,v_n,0,\dots,0).
$$
Thus, for even $j > n/2$, we have (also see
\autoref{eqn:upper_bound_chi_q_Hamming} and \cite[Theorem 1.1]{SDU})
$$
\chi_q(H(n,j)) \le \chi_q(H(2j,j)) = 2j.
$$
For $p = 4$ and even $n$, note that
$$
O_{n,4} \cong \bigcup_{r=0}^{n/2} G(r,n/2-r).
$$
In particular, $G(r,s)$ is a subgraph of $O_{2(r+s),4}$. So,
\begin{equation}
	\label{eqn:chiq_ub_4}
\chi_q(G(r,s)) \le \chi_q(O_{2(r+s),4}) \le 2(r+s).
\end{equation}
There are some other works focusing on the quantum chromatic number of
subgraphs of $O_{n,p}$ \cite{Feng,OUR} (also see \autoref{tab:results}).

\subsection{Lower Bounds on Quantum Chromatic Numbers}
In this subsection, we apply the spectral lower bound (\autoref{eqn:spectral_lb})
together with the results on the smallest eigenvalue in the previous sections 
to derive lower bounds on $\chi_q(H(n,j))$ and $\chi_q(G(r,s))$.
\begin{corollary}
\label{corollary:lb_chiq_Hamming}
Let $j$ be even and let $n$ be large enough.
$$
\chi_q(H(n,j)) = \begin{cases}
	\Omega(n^{j/2}), & j = O(1), \\
	\Omega\qty(\qty(
		\frac{2^{h(\alpha)}}{
			\left(e(1+\sqrt{2})\sqrt{(1-\alpha)/\alpha}\right)^\alpha
		}
	)^n), & j = \alpha n,
\end{cases}
$$
where $0 < \alpha \le 0.17$.
\end{corollary}
\begin{proof}
The result follows from \autoref{corollary:lb_Hamming} and
\autoref{eqn:spectral_lb} by noting that $\lambda_{\max}(H(n,j)) = \binom{n}{j}$.
\end{proof}

Note that when $j = \alpha n$, we have an upper bound of $\chi_q(H(n,j))$ in
\autoref{eqn:upper_bound_chi_q_Hamming}, 
$$\chi_q(H(n,j))\leq  2^{h\qty(\frac{1}{2}-\sqrt{\alpha(1-\alpha)})n + o(n)},~ 0
< \alpha < 1/2.$$ We compare the lower bound in
\autoref{corollary:lb_chiq_Hamming} with this upper bound. Let 
\begin{equation}
	\label{eqn:lu}
l(\alpha) \defeq \frac{2^{h(\alpha)}}{\left(e(1+\sqrt{2})\sqrt{(1-\alpha)/\alpha}\right)^\alpha}, 
\qand u(\alpha) \defeq 2^{h\left(1/2-\sqrt{\alpha(1-\alpha)}\right)}
\end{equation}
be the  bases of the  two exponential lower(upper) bounds. Some numerical values
of $l(\alpha)$ and $u(\alpha)$ for $\alpha\leq 0.17$ are listed in
\autoref{tab:compare} for comparison. We note that when $\alpha$ is approaching
$0.17$, the gap between $l(\alpha)$ and $u(\alpha)$  becomes smaller.

\begin{corollary}
	\label{corollary:chi4}
Let $n = 2(r+s)$ with $r \ge s$ and let $r$ be even. Then
$$
\chi_q(G(r,s)) = n,
$$
in both of the following two regimes
\begin{enumerate}
	\item[(R1).] $s \ge 2$ is fixed and $n$ is large enough, i.e. $s = O(1)$;
	\item[(R2).] $s = \alpha \cdot \frac{n}{2}$ and $n$ is large enough, 
	where $0 < \alpha \le 0.17$.
\end{enumerate}
\end{corollary}
\begin{proof}
By \autoref{eqn:spectral_lb} and
\autoref{theorem:smallest_eigenvalue_4}, $\chi_q(G(r,s)) \ge
n$, in both (R1) and (R2). On the other hand, it is known in
\autoref{eqn:chiq_ub_4}, that $\chi_q(G(r,s)) \le n$.
\end{proof}

\begin{table}
	\centering
\begin{threeparttable}
	\caption{Comparison between the lower and upper bounds on $\chi_q(H(n,\alpha n))$}
	\label{tab:compare}
\begin{tabular}{l|lllllllll}
\toprule
$\alpha$    &$0.01$ & $0.02$ & $0.03$ & $0.04$ & $0.05$ & $0.06$ & $0.07$
& $0.08$ & $0.09$ \\
$l(\alpha)$ &$1.062$ & $1.105$ & $1.140$ & $1.171$ & $1.198$ & $1.222$ & $1.243$
& $1.262$ & $1.279$ \\
$u(\alpha)$ &$1.961$ & $1.922$ & $1.885$ & $1.848$ & $1.813$ & $1.778$ & $1.745$
& $1.712$ & $1.680$ \\
\midrule
$\alpha$ & $0.10$ & $0.11$ & $0.12$ & $0.13$ & $0.14$ & $0.15$ & $0.16$ & $0.17$
\\
$l(\alpha)$ & $1.293$ & $1.307$ & $1.318$ & $1.328$ & $1.336$ & $1.343$ &
$1.349$ & $1.353$ \\
$u(\alpha)$ & $1.649$ & $1.619$ & $1.590$ & $1.562$ & $1.534$ & $1.507$ &
$1.481$ & $1.456$ \\
\bottomrule
\end{tabular}
\begin{tablenotes}
\item Here, $\Omega(l(\alpha)^n) \leq \chi_q(H(n,\alpha n)) \le
u(\alpha)^{n+o(n)}$ when $n$ is large enough.
\end{tablenotes}
\end{threeparttable}
\end{table}

\section{Conclusion}
\label{sec:conclusion}
We provide asymptotic lower bounds for the smallest eigenvalue and the quantum
chromatic number of the distance-$j$ Hamming graph $H(n,j)$ for even $j < n/2$.
In the regime $j = O(1)$, we have $|\lambda_{\min}(H(n,j))| = O(n^{j/2})$, which
implies that $\chi_q(H(n,j)) = \Omega(n^{j/2})$. The regime $j = \Theta(n)$ is
also considered, and exponential lower bounds on $\lambda_{\min}(H(n,j))$ and
$\chi_q(H(n,j))$ are derived. By these lower bounds on the smallest eigenvalue
of $H(n,j)$, we also showed that $\lambda_{\min}(\cay(\ZZ_4^n,(r,s,r,s))) = -
\frac{\binom{n}{r,s,r,s}}{n-1}$ and $\chi_q(\cay(\ZZ_4^n,(r,s,r,s))) = n$ with
$n = 2(r+s)$, for fixed $s\geq 2$ or $s=\frac{1}{2}\alpha n$ with $0 < \alpha
\le 0.17$, where $r \ge s$ is even and large enough.

There are many related problems left. The following ones might be interesting
for future study.
\begin{enumerate}[label=(\arabic*)]
\item Find an upper bound for $\chi_q(H(n,j))$ in the regime $j = O(1)$. The case
$j = 2$ has been considered in \cite{OUR} (also see \autoref{tab:results}).
\item The approximation in \autoref{theorem:lb_q_theta} is rough. Can this be refined
to provide better lower bounds on $\lambda_{\min}(H(n,j))$?
\item There is still a gap between the lower bound of $\chi_q(H(n,j))$ in this paper
and the upper bound in \autoref{eqn:upper_bound_chi_q_Hamming}, in the regime
$j = \Theta(n)$. A solution to Problem (2) might narrow this gap.
\item We believe that \autoref{theorem:smallest_eigenvalue_4} holds for all
$n=2(r+s)$ with even $r \ge s$, not just the regimes considered in this
paper. If this is true, it will imply that $\chi_q(G(r,s)) = n$ with $n = 2(r+s)$ for all
even $r \ge s$.
\item Finally, what is the smallest eigenvalue of the graph $H(n,j)$ for even $j < n/2$?
\end{enumerate}

\section*{Acknowledgements}
Yu Ning and Xiande Zhang are supported in part by the National Key Research and 
Development Programs of China under Grant 2023YFA1010200,  in part by NSFC 
under Grant 12171452 and Grant 12231014, and in part by the Quantum Science and 
Technology-National Science and Technology Major Project under Grant 
2021ZD0302902.
Jack H. Koolen is partially supported by NSFC under Grant 12471335.
The authors thank Tao Luo for helpful discussions.

\appendix
\section{Fibonacci Polynomials}
\label{appsec:Fibonacci}
We collect some results on the Fibonacci polynomial
\cite{FibonacciPolynomialWebsite}, which are used in the proof of
\autoref{theorem:lb_q_theta}. The Fibonacci polynomial $F_n(x)$ is defined by
the recursion $F_{n+1}(x) = xF_n(x) + F_{n-1}(x)$ with initial values $F_1(x) =
1$ and $F_2(x) = x$. It is known that
\cite{FibonacciPolynomialWebsite}
\begin{equation}
	\label{eqn:Fibonacci}
F_n(x) = \sum_{k=0}^{\floor{(n-1)/2}} \binom{n-1-k}{k}x^{n-1-2k} =
\frac{\qty(x+\sqrt{x^2+4})^n - \qty(x-\sqrt{x^2+4})^n}
{2^n\sqrt{x^2+4}}.
\end{equation}
The following related polynomial is directly used in the proof of
\autoref{theorem:lb_q_theta},
\begin{equation}
f_n(x) \defeq \sum_{k=0}^{\floor{n/2}} \binom{n-k}{k}x^k.
\end{equation}
By \autoref{eqn:Fibonacci}, $f_n(x) = \sqrt{x}^n F_{n+1}(1/\sqrt{x})$,
and we have
\begin{equation}
	\label{eqn:f2}
f_n(x) = \frac{(1+\sqrt{1+4x})^{n+1} - (1-\sqrt{1+4x})^{n+1}}
{2^{n+1}\sqrt{1+4x}}.
\end{equation}

\bibliographystyle{plain}
\bibliography{bib.bib}
\end{document}